\documentclass [11pt] {article}
\usepackage{amsfonts}
\usepackage{amssymb}
\usepackage{times}
\usepackage{color}
\usepackage{graphicx}

\newtheorem{lemma}{Lemma}[section]
\newtheorem{corollary}{Corollary}[section]
\newtheorem{theorem}{Theorem}

\newtheorem{conjecture}{Conjecture}
\newtheorem{proposition}{Proposition}[section]
\newtheorem{definition}{Definition}[section]

\def\<{\langle}
\def\>{\rangle}

\begin{document}

\title{On the logarithimic calculus and Sidorenko's conjecture}
\author{{\sc J.L. Xiang Li~~~~~ Bal\'azs Szegedy}}

\maketitle

\abstract{We study a type of calculus for proving inequalities between subgraph densities which is based on Jensen's inequality for the logarithmic function. As a demonstration of the method we verify the conjecture of Erd\"os-Simonovits and Sidorenko for new families of graphs. In particular we give a short analytic proof for a result by Conlon, Fox and Sudakov. Using this, we prove the forcing conjecture for bipartite graphs in which one vertex is complete to the other side.}

\bigskip
\section{Introduction}
Inequalities between subgraph densities is subject of extensive study. Many problems in extremal graph theory can be formulated in this language.  Subgraph densities can be conveniently defined through graph homomorphisms.  A graph homomorphism between finite graphs $H=(V(H),E(H))$ and $G=(V(G),E(G))$ is a map $\phi:V(H)\rightarrow V(G)$ such that the image of every edge in $H$ is an edge in $G$.  The subgraph density $t(H,G)$ is the probability that a random map $\phi:V(H)\rightarrow V(G)$ is a graph homomorphism. 

In the frame of the graph limit theory \cite{LSz} $G$ can be replaced by an analytic object which is a two variable symmetric measurable function.
Many statements in graph theory can be equivalently stated in this analytic language as follows. 
Let $(\Omega,\mu)$ be a probability space and $W:\Omega^2\rightarrow\mathbb{R}$ be a bounded measurable function such that $W(x,y)=W(y,x)$ for every pair $x,y\in\Omega$. If the range of $W$ is in the interval $[0,1]$ then it is called a {\it graphon}. Let 
\begin{equation}\label{dens}
t(H,W)=\mathbb{E}\bigl(\prod_{(x_i,x_j)\in E(H)}W(x_i,x_j)\bigr)
\end{equation}
where $V(H)=\{x_i\}_{i=1}^n$ are independently chosen from $\Omega$.  
It is easy to see that if $\Omega=V(G)$ with the uniform distribution and $W_G:V(G)\times V(G)\rightarrow\{0,1\}$ is the adjacency matrix of $G$ then $t(H,G)=t(H,W_G)$. 

The Cauchy-Schwartz inequality is a fundamental tool in establishing inequalities between subgraph densities.  It was shown in \cite{LSz} that every true inequality between subgraph densities is a consequence of a possibly infinite number of Cauchy-Schwartz inequalities. As a negative result Hatami and Norin \cite{HN} proved the undecidability of the general problem. 

Motivated by earlier work on entropy calculations \cite{KR} in this area, we develop a method using Jensen's inequality for logarithmic functions. In particular we use the concavity of $z\mapsto\ln z$ and the convexity of $z\mapsto z\ln z$.
Similarly to the Cauchy-Schwartz calculus, the logarithmic calculus is capable of proving inequalities through a line of symbolic calculations without verbal arguments. This can make it suitable for computer based algorithms designed to find new inequalities.

We demonstrate the power of the logarithmic calculus approach on a fundamental problem known as Sidorenko's conjecture (also conjectured by Erd\"os and Simonovits \cite{Sim}). In the combinatorial language the conjecture says that following. 

\begin{conjecture} Let $H$ be a bipartite graph and $G$ be an arbitrary graph. Then $t(H,G)\geq t(P_2,G)^e$ where $P_2$ is the single edge and $e$ is the number of edges in $H$. 
\end{conjecture}

Phrased originally by Sidorenko in the analytic language \cite{Sid}, the conjecture says that $t(H,W)\geq t(P_2,W)^e$ holds for every bounded non-negative measurable function $W$. Expressions of the form $t(H,W)$ appear as Mayer integrals in statistical mechanics, Feynman integrals in quantum field theory and multicenter integrals in quantum chemistry. The conjecture is also related to many topics such as Markov chains \cite{BP}, matrix theory and quasi-randomness. 
Note that in the analytic form if $W$ is constant then $t(H,W)=t(P_2,W)^e$. The forcing conjecture rephrased in this analytic language says that if $H$ is not a tree then this is the only case when equality holds. Note that this is a refinement of Sidorenko's conjecture.

Sidorenko's conjecture has been verified for various families of graphs, including even cycles, trees \cite{Sid}, and hypercubes \cite{Hat}.  Most recently the conjecture has been proven by Sudakov, Conlon and Fox \cite{CFS} for bipartite graphs with one vertex complete to the other side.  The proof uses the tensor power trick and a certain probabilistic method called dependent random choice.  As a quick demonstration of the logorithmic calculus we give an analytic proof of this result. 
Essentially the same proof also implies the forcing conjecture for the same graphs. 
Note that the forcing conjecture for bipartite graphs in which two vertices are complete to the other side was proved by Conlon, Sudakov and Fox \cite{CFS}.

The logarithmic calculus also yields stronger intermediate inequalities that are preserved under `gluing operations' and thus provide an iterative method for proving Sidorenko's conjecture for many new graphs. A theorem of this type is the following. We call a graph a reflection tree (for precise definition see chapter \ref{smr}) if it can be obtained from a tree by gluing reflected versions of subtrees on it.
They include the bipartite graphs considered by Conlon, Fox and Sudakov. 
Let us call a bipartite graph $H$ a {\it Sidorenko graph} if $t(H,W)\geq t(P_2,W)^e$ for every graphon $W$. 

\begin{theorem}\label{rts} Reflection trees are Sidorenko.
\end{theorem}

Another theorem that we prove is the following.

\begin{theorem}[Edge gluing]\label{edglue} Let $H_1, H_2$ be Sidorenko graphs and $e_1\in E(H_1),e_2\in E(H_2)$ be arbitrary edges. Then the graph $H$ obtained from $H_1$ and $H_2$ by identifying $e_1$ and $e_2$ is also a Sidorenko graph.  
\end{theorem}

Note that in the above theorem the edges $e_1$ and $e_2$ can be identified in two different ways.

\section{Basics of Logarithmic Calculus}

\begin{lemma}[Jensen's inequality] Let $(\Omega,\mu)$ be a probability space, let $c$ be a convex (resp. concave) function on an interval $D\subset\mathbb{R}$ and $g:\Omega\rightarrow D$ be a measurable function. Then 
\begin{equation}\label{jen1}
\mathbb{E}(c(g))\geq c(\mathbb{E}(g))~~{\rm{\it (convex)}}~~~,~~~\mathbb{E}(c(g))\leq c(\mathbb{E}(g))~~{\rm{\it (concave)}}
\end{equation}
 Moreover if $\mathbb{E}(f)=1$ for some non-negative function $f$ on $\Omega$ then also 
 \begin{equation}\label{jen2}
\mathbb{E}(fc(g))\geq c(\mathbb{E}(fg))~~{\rm{\it (convex)}}~~~,~~~\mathbb{E}(fc(g))\leq c(\mathbb{E}(fg))~~{\rm{\it (concave)}}
\end{equation}
If $c$ is a strictly convex (concave) function then equality in (\ref{jen2}) is only possible if $g$ is constant on the support of $f$.
\end{lemma}

\begin{proof} Jensen's inequality is a classical result whose proof is based on the intuitively clear fact that if a probability measure is concentrated on a convex (resp. concave) curve then the center of mass is above (resp. below) the curve.
The inequality (\ref{jen2}) is a direct consequence of (\ref{jen1}) if we consider $f$ as the density function of a new measure $\mu^*$ on $\Omega$. 
\end{proof}

\medskip

\medskip

We introduce some notation. Let $g(x_1,x_2,\dots,x_n)$ be a function on $[0,1]^n$. For a subset $S\subseteq\{x_1,x_2,\dots,x_n\}$ of the variables the we introduce the $|S|$-variable function
$$\mathbb{E}_S(g)=\int g~\prod_{x_i\in \overline{S}}dx_i.$$
Notice that $\mathbb{E}(g)=\mathbb{E}(\mathbb{E}_S(g))$. To simplify notation we will identify vertices of graphs with variables representing them in the formula (\ref{dens}). As the first illustration of the logarithmic calculus we start by the Blakley-Roy inequality. 

\begin{proposition}[Blakley-Roy] Let $W:[0,1]^2\rightarrow\mathbb{R}^+$ be a bounded symmetric measurable function and $d=\mathbb{E}(W)$. Then $t(P_n,W)\geq d^n$. 
\end{proposition}

\begin{proof} Let $d(x)=\mathbb{E}_x(W(x,y))$ and $$f=\Bigl(\prod_{i=1}^n W(x_i,x_{i+1})\Bigr)d^{-1}\Bigl(\prod_{i=2}^n d(x_i)^{-1}\Bigr).$$ We have that $\mathbb{E}(f)=1$.
$$\ln t(P_n,W)=\ln\mathbb{E}(fd\prod_{i=2}^n d(x_i))\geq\ln d+\mathbb{E}(f\ln\prod_{i=2}^nd(x_i))=\ln d+\sum_{i=2}^n\mathbb{E}(f\ln d(x_i))=$$
$$=\ln d+\sum_{i=2}^n\mathbb{E}(\mathbb{E}_{x_i}(f\ln d(x_i)))=\ln d+\sum_{i=2}^nd^{-1}\mathbb{E}(d(x_i)\ln d(x_i))\geq n\ln d.$$
The first inequality follows from (\ref{jen2}) with $c(z)=\ln z$ and the second inequality follows from (\ref{jen1}) with $c(z)=z\ln z$.
\end{proof}

\medskip

\begin{theorem}[Conlon-Fox-Sudakov]\label{CFST} Let $H$ be the (bipartite) graph on the vertex set $\{x,y_1,y_2,\dots,y_m,v_1,v_2,\dots,v_k\}$ such that $x$ is connected to $v_1,v_2,\dots ,v_k$ and $y_t$ is connected to the vertices $S_t\subset\{v_1,v_2,\dots,v_k\}$ where $|S_t|=a_t$.   
Let $e=k+\sum_{t=1}^m a_t$ be the total number of edges in $H$. Then if $W:[0,1]^2\rightarrow\mathbb{R}^+$ is a measurable function and $d=\mathbb{E}(W)$. Then $t(H,W)\geq d^e.$
\end{theorem}

\begin{proof} Let $$q=\mathbb{E}_x(W(x,z))~~~{\rm and}~~~f=d^{-1}q^{1-k}\prod_{i=1}^kW(x,v_i),$$ $$s_t=\mathbb{E}_{S_t}\bigl(\prod_{v_j\in S_t}W(z,v_j)\bigr)~~~{\rm and}~~~f_t=s_t^{-1}q^{a_t-k}\prod_{i=1}^k W(x,v_i)~~~(t=1,2,\dots,m).$$ Notice that $\mathbb{E}(f)=\mathbb{E}(f_t)=1$. Using (\ref{jen2}) with $c(z)=\ln z$ and (\ref{jen1}) with $c(z)=z\ln z$ we have $$\ln t(H,W)=\ln\mathbb{E}(f q^{k-1}d\prod_{i=1}^m s_i)\geq \mathbb{E}( f \ln (q^{k-1}d\prod_{i=1}^m s_i))=$$ $$=(k-1)\mathbb{E}(\mathbb{E}_{x}(f\ln q))+\mathbb{E}(f\ln d)+\sum_{i=1}^m\mathbb{E}(f\ln s_i)=$$
$$=(k-1)d^{-1}\mathbb{E}(q\ln q)+\ln d+\sum_{i=1}^m\mathbb{E}(f\ln s_i)\geq k\ln d+\sum_{i=1}^m\mathbb{E}(f\ln s_i).$$
Let $h_t=s_t q^{1-a_t}$. Then by $\mathbb{E}_{x}(f_t h_t)=q$, $\mathbb{E}(f_t h_t)=d$ and (\ref{jen2}),(\ref{jen1}) with $c(z)=z\ln z$, $$\mathbb{E}(f\ln s_t)=d^{-1}\mathbb{E}(f_t h_t\ln h_t)+(a_t-1)d^{-1}\mathbb{E}(\mathbb{E}_{x}(f_t h_t\ln q))=$$
$$=d^{-1}\Bigl(\mathbb{E}(f_th_t\ln h_t)+(a_t-1)\mathbb{E}(q\ln q)\Bigr)\geq a_t\ln d.$$
\end{proof}

\medskip

\noindent{\bf Remark:}~In the proof of theorem \ref{CFST} it is convenient to assume that $W$ is strictly positive. Then using the continuity of $t(H,W)$ under $L^1$ convergence we obtain the Sidorenko inequality for every non-negative $W$. However it is a useful observation that with some extra caution all the steps of the proof work for non-negative functions. The reason for this is that if we replace all the weight function $f$ and $f_t$ occurring in the proof by modified versions $f'$ and $f_t'$ in which values that are not defined are replaced by $0$ then the same calculations remain true. This is a simple case by case checking.

\medskip 

\begin{theorem}[Forcing for C-F-S graphs] The forcing conjecture holds for bipartite graphs in which one vertex is complete to the other side (and are not trees).
\end{theorem}

\begin{proof} We use the notation from the proof of theorem \ref{CFST} . In this proof we allow $W$ to take the value $0$. Let $f'$ and $f'_t$ be as in the remark following the proof of Theorem \ref{CFST}. Assume that $t(H,W)=d^e$. Then in the proof of theorem \ref{CFST} all the inequalities become equalities. 
We immediately obtain that $\mathbb{E}(q\ln q)=d\ln d$ and thus $q$ has to be constant $d$ on $[0,1]$.
If $H$ is not a tree then there is a number $1\leq j\leq k$ such that $a_j\geq 2$.
We obtain that $\mathbb{E}(f'_jh_j\ln h_j)=d\ln d$ and thus $h_j$ has to be constant $d$ on the support of $f'_j$.  This implies that $s_j$ is constant $d^{a_t}$ on the support of $f'_j$ which is equal to the support of $\prod_{i=1}^kW(x,v_i)$.
This implies that if the value of $s_j$ is not $0$ then it is equal to $d^{a_t}$. On the other hand (since spiders are Sidorenko) we have $\mathbb{E}(s_t)\geq d^{a_t}$ and thus $s_j$ is constant $d^{a_t}$. This implies that $t(K_{2,a_t},W)=\mathbb{E}(s_t^2)=d^{2a_t}$. Using that $K_{2,a_t}$ is forcing the proof is complete.  
\end{proof}

\section{Smoothness and Gluing}\label{smr}

In general, we consider the logarithmic calculus as a symbolic way of proving inequalities between subgraph densities using conditional expectations and Jensen's inequality for  $z\ln z$ and $\ln z$.
In this chapter we give a demonstration by proving Sidorenko's conjecture for a family of bipartite graphs.
However this is not the limitation of the method inside Sidorenko's conjecture and there are many applications outside Sidorenko's conjecture. Further applications will be discussed in a subsequent paper.

The proof of theorem \ref{CFST} relies on the interesting phenomenon that certain bipartite graphs satisfy stronger inequalities than Sidorenko's.  Note that the form of these new types of inequalities hints at useful reinterpretation as geometric averages (detailed later).  More specifically we showed that if $H$ is the graph on the vertex set $\{x, v_1,...,v_k, y\}$, where $x$ is connected to $v_1,\dots ,v_k$ and $y$ is connected to a subset $S$ of $\{v_1,v_2,\dots,v_k\}$ then it satisfies the inequality
\begin{equation}\label{smooth1}
\mathbb{E}\Bigl(d^{-1}q^{n-1}\prod_{i=1}^nW(x,v_i)\ln s\Bigr)\geq |S|\ln d
\end{equation}
where $s=\mathbb{E}_S\Bigl(\prod_{v_j\in S}W(y,v_j)\bigr)$ and $d$ and $q$ are as in the proof of the theorem.
This inequality (\ref{smooth1}) allows one to glue together such graphs $H$ with various choices of $S$ along the star spanned on $\{x,v_1,v_2,\dots,v_k\}$ and the resulting graph will still be Sidorenko.  
Through this operation we can build the type of graphs considered by Sudakov, Conlon and Fox.  This indicates that (\ref{smooth1}) type inequalities can be used to produce new Sidorenko graphs by gluing.  Our goal is to generalize this situation.  For ease of expression, we first introduce restricted subgraph densities.  

Let $\mathcal{G}_n$ denote the set of graphs in which $n$ different vertices are labeled by the numbers $\{1,2,\dots,n\}$.
If $H_1$ and $H_2$ are in $\mathcal{G}_n$ then their product $H_1H_2$ is defined as the graph obtained form them by identifying vertices with the same label and then reducing multiple edges. The notion of subgraph density can be naturally extended for graphs in $\mathcal{G}_n$. 

\begin{definition}[Restricted subgraph density]
Let $H\in\mathcal{G}_n$ be a graph on the vertex set $\{x_1,x_2,\dots,x_m\}$ such that the labeled vertices are $S=\{x_1,x_2,\dots,x_n\}$. Then the restricted subgraph density of $H$ in $W:[0,1]^2\rightarrow\mathbb{R}$ is the $n$-variable function defined as
\begin{equation}\label{dens2}
t_S(H,W)=\mathbb{E}_S\bigl(\prod_{(x_i,x_j)\in E(H)}W(x_i,x_j)\bigr).
\end{equation}  
\end{definition}

\begin{lemma}\label{treeweight} Let $T$ be a tree on the vertex set $\{x_1,x_2,\dots,x_n\}$ and let 
\begin{equation}\label{weight}
f_T=d^{-1}\prod_{i=1}^n d(x_i)^{1-r_i}\prod_{(x_i,x_j)\in E(T)}W(x_i,x_j).
\end{equation}
Then for an arbitrary $1\leq i\leq n$ we have that $\mathbb{E}_{x_i}(f_T)=d(x_i)/d$ and consequently $\mathbb{E}(f_T)=1$.
\end{lemma}

\begin{proof} If $n=1$ then the statement is trivial. By induction assume that it holds for $n-1$. Assume that $n>1$ and $x_j\neq x_i$ is a leaf in $T$. Let $S=\{x_1,\dots,x_n\}\setminus\{x_j\}$. Then $\mathbb{E}_{x_i}(f_T)=\mathbb{E}_{x_i}(\mathbb{E}_S(f_{T}))=\mathbb{E}_{x_i}(f_{T'})$ where $T'$ is obtained from $T$ by deleting $x_j$ and thus the induction step finishes the proof.  
\end{proof}

We are now ready to define the notion of smoothness. 

\begin{definition}[Smoothness] Let $H\in\mathcal{G}_n$ be a bipartite graph on the vertex set $\{x_1,x_2,\dots,x_m\}$ such that the spanned subgraph on $S=\{x_1,x_2,\dots,x_n\}$ is a tree $T$. We say that $H$ is smooth (or $T$ is smooth in $H$) if
$$\mathbb{E}\Bigl(f_T~\ln t_S(H^*,W)\Bigr)\geq |E(H^*)|\ln d$$
where $H^*$ is the graph obtained from $H$ by deleting the edges in $T$. 
\end{definition}

Note that if $n=0$ then $T$ is empty. In this case the smoothness of $H$ is equivalent with the Sidorenko property. The next two lemmas together show that smoothness in general, is a strengthening of the Sidorenko property. 

\begin{lemma}[Unlabeling]\label{unlab} Let $H\in\mathcal{G}_n$ be a smooth bipartite graph with tree $T$ spanned on the labeled vertices $S$ and let $T'$ be a non-empty sub-tree spanned on $S'\subset S$. Then the graph $H_2$ obtained from $H$ by unlabeling the vertices in $S\setminus S'$ is also smooth. 
\end{lemma} 

\begin{proof} It is enough to prove that unlabeling one leaf in $H$ preserves smoothness. Every other case can be obtained by iterating this step. Assume that $x_n$ is a leaf connected to $X_{n-1}$.
We have that
$$\mathbb{E}\Bigl(f_{T'}~\ln t_{S'}(H_2^*,W)\Bigr)=\mathbb{E}\Bigl(f_{T'}~\ln\mathbb{E}_{x_n}(W(x_n,x_{n-1})t_S(H^*,W))\Bigr)=$$
$$=\mathbb{E}\Bigl(f_{T'}~\ln\mathbb{E}_{x_n}(W(x_n,x_{n-1})d(x_{n-1})^{-1}d(x_{n-1})t_S(H^*,W))\Bigr)\geq$$
$$\mathbb{E}(f_{T'}~\ln d(x_{n-1}))+\mathbb{E}(f_{T}~\ln t_S(H^*,W))\geq\mathbb{E}(f_{T'}~\ln d(x_{n-1}))+|E(H^*)|\ln d.$$
In the above calculation we use the concavity of $z\rightarrow\ln (z)$ with weight function $W(x_n,x_{n-1})d(x_{n-1})^{-1}$.
Using lemma \ref{treeweight} we get 
$$\mathbb{E}(f_{T'}~\ln d(x_{n-1}))=\mathbb{E}(\mathbb{E}_{x_{n-1}}(f_{T'}~\ln d(x_{n-1})))=d^{-1}\mathbb{E}(d(x_{n-1})\ln d(x_{n-1}))\geq\ln d.$$
\end{proof}

\begin{lemma}\label{edgesid} Assume that $(x_1,x_2)\in E(H)$ and $H\in\mathcal{G}_2$ is smooth. Then $H$ is a Sidorenko graph.
\end{lemma}

\begin{proof} We have that $$\ln\mathbb{E}(t(H,W))=\ln d+\ln\mathbb{E}(W(x_1,x_2)d^{-1}t(H^*,W))\geq$$
$$\geq \ln d+\mathbb{E}(W(x_1,x_2)d^{-1}\ln t(H^*,W))\geq \ln d+|E(H^*)|\ln d=|E(H)|\ln d.$$
\end{proof}

Lemma \ref{unlab} and lemma \ref{edgesid} together imply the following important corollary.

\begin{corollary} Let $H\in\mathcal{G}_n$ be a smooth graph. Then $H$ is Sidorenko.
\end{corollary}

The next lemma shows how to produce now Sidorenko graphs from old ones using smoothness. 

\begin{lemma}[Gluing on smooth trees]\label{gluing} Let $H_1,H_2\in\mathcal{G}_n$ be two smooth graphs such that the trees spanned on the labeled vertices are identical with $T$ in both graphs. Then $H=H_1H_2\in\mathcal{G}_n$ is also smooth.
\end{lemma}

\begin{proof} 
$$\mathbb{E}(f_T\ln t_S(H^*,W))=\mathbb{E}(f_T\ln t_S(H_1^*,W)t_S(H_2^*,W))=\mathbb{E}(f_T\ln t_S(H_1^*,W))+\mathbb{E}(f_T\ln t_S(H_2^*,W))$$
$$\geq (|E(H_1)|-n+1)\ln d+(|E(H_2)|-n+1)\ln d=(|E(H_1H_2)|-n+1)\ln d.$$
\end{proof}

\begin{lemma}[Extension of smooth part]\label{ext} Let $H\in\mathcal{G}_n$, $(n>1)$ be a smooth graph and let $T'\in\mathcal{G}_n$ be a tree such that $H$ and $T'$ induce the same tree $T$ on the labeled points. Then the graph $H_2$ obtained from $HT'$ by putting labels on all the vertices in $T'$ is again smooth.
\end{lemma}

\begin{proof} It is enough to prove the statement for the case where $y=V(T')\setminus V(T)$ is a single vertex which is a leaf in $T'$. The general case is an iteration of this step. Without loss of generality assume that $y$ is connected to $x_n$.
Let $S=\{x_1,x_2,\dots,x_n\}$, $S'=S\cup\{y\}$ and $s=t_{S'}(H_2^*,W)$.
In this case
$$\mathbb{E}(f_{T'}\ln s)=\mathbb{E}(\mathbb{E}_S(f_{T'}\ln s))=\mathbb{E}(f_T~\ln s)=\mathbb{E}(f_T~\ln~t_S(H^*,W))\geq$$
$$\geq |H^*|\ln d=|H_2^*|\ln d.$$
\end{proof}

\begin{definition}[Reflection] Reflection of a subgraph $H_2$ induced on $K\subset V(H)$ in $H$ along an independent set $S\subset K$ is the operation which produces the graph $HH_2$ where the vertices in $K$ are labeled in both $H$ and $H_2$ in the same way.
\end{definition}

\begin{lemma}\label{refl} Let $T\in\mathcal{G}_n$ be a tree such that the labeled points $S$ are independent. Let $H$ be the graph obtained from $T^2$ by labeling all points in one copy of $T$. Then $H$ is smooth.
\end{lemma}

\begin{proof} The statement is obviously equivalent with the following one. Let $T$ be a tree on $M=\{x_1,x_2,\dots,x_m\}$ and let $S=\{x_1,x_2,\dots,x_n\}\subset M$. Then 
$$\mathbb{E}(f_T\ln t_S(T,W))\geq |E(T)|\ln d.$$
Let $s=t_S(T,W)$ and  $q=\prod_{i=1}^m d(x_i)^{r_i-1}$. Then
$$\mathbb{E}(f_T\ln s)=d^{-1}\mathbb{E}(t_M(T,W)s^{-1}~(sq^{-1})\ln (sq^{-1}))+d^{-1}\mathbb{E}(t_M(T,W)s^{-1}sq^{-1}\ln q).$$
$$\geq \ln d+\sum_{i=1}^m d^{-1}(r_i-1)\mathbb{E}(\mathbb{E}_{x_i}(t_M(T,W)s^{-1}sq^{-1}\ln d(x_i)))=$$
$$=\ln d+\sum_{i=1}^m d^{-1}(r_i-1)\mathbb{E}(d(x_i)\ln d(x_i))\geq\ln d+\sum_{i=1}^m (r_i-1)\ln d=(m-1)\ln d.$$
\end{proof}

To illustrate lemma \ref{refl} on an example, let $T$ be the path of length $m$ such that the two endpoints are labeled. Then $T^2$ is the cycle $C_{2m}$ of length $2m$. The lemma implies that a path of length $m$ is smooth inside $C_{2m}$. Then lemma \ref{unlab} shows that any path of length at most $m$ is also smooth inside $C_{2m}$ and thus in particular edges are smooth. It follows from Lemma \ref{gluing} that if we glue together even cycles along an edge then the resulting graph is Sidorenko. We will see later that this is true with arbitrary Sidorenko graphs.

Now we introduce a class of graphs and we call them {\it reflection trees}. They include trees, even cycles, and bipartite graphs in which one vertex is complete to the other side. We prove that reflection trees are Sidorenko.

\begin{definition}[Reflection tree] A reflection tree is a graph obtained from a tree $T$ by applying the reflection operation to a collection of sub-trees in $T$. 
\end{definition}

Two examples for reflection trees are even cycles and bipartite graphs in which one vertex is complete to the other side.
Even cycles are obtained by reflecting a path and the other example is obtained from a star reflecting sub-stars.
Now we are ready to prove that reflection trees are Sidorenko.

\medskip

\noindent{\it Proof of Theorem \ref{rts}.}~~ Lemma \ref{refl} and lemma \ref{ext} together show that in any graph obtained from a tree $T$ by reflecting a subtree the tree $T$ is smooth. Then lemma \ref{gluing} finishes the proof.

\bigskip

\section{Smoothness of edges}

\bigskip

To avoid complications, our original definition for smoothness used functions $W$ that are strictly positive.  However smoothness can be equivalently defined in purely graph theoretic terms.
Using the notation from the previous chapter let $H\in\mathcal{G}_n$ be a bipartite graph on the vertex set $\{x_1,x_2,\dots,x_m\}$ such that the spanned subgraph on $S=\{x_1,x_2,\dots,x_n\}$ is a tree $T$. Let $G$ be a finite graph (replacing $W$). Homomorphisms from $H$ (resp. $T$) to $G$ can be looked at as assignments of values from $V(G)$ to the variables $\{x_1,x_2,\dots,x_m\}$ (resp $\{x_1,x_2,\dots,x_n\}$). Let $\mu_T$ denote the probability distribution on $V(G)^n$ defined by the density function $f_T$. The places where $f_T$ is not defined are set to $0$. Intuitively, the distribution $\mu_T$ is a random copy of $T$ built up in $G$ by first choosing a random edge $e$ and then growing $T$ by adding leaves to the existing configuration one by one in a random way. This can also be considered a branching random walk with structure $T$ started from a random edge.
The inequality defining smoothness becomes
\begin{equation}\label{smooth2}
\mathbb{E}(\ln t_S(H^*,G))\geq |E(H^*)|\ln d.
\end{equation}
where the expected value is taken with respect to the distribution $\mu_T$.
It can be seen from standard methods in graph limit theory that the two definitions are equivalent. 
From (\ref{smooth2}) we can immediately see that for every (homomorphic) copy of $T$ in $G$ the restricted homomorphism density $t_S(H^*,G)$ can't be $0$. This gives an interesting topological obstruction to smoothness.
 
\begin{definition}[Retract] Let $H$ be a graph and $S\subset V(H)$ be a subset of its vertices. We say that $S$ (Or the graph $H_2$ spanned on $S$) is a retract of $H$ if there is a graph homomorphism $\phi:H\rightarrow H_2$ such that $\phi$ restricted to $H_2$ is the identity map.
\end{definition}

It is easy to see that $H_2$ is a retract if and only if any graph homomorphism $\phi':H_2\rightarrow G$ into an arbitrary graph $G$ extends to a graph homomorphism $\phi:H\rightarrow G$.
The next lemma follows immediately from (\ref{smooth2}).

\begin{lemma} If a tree $T$ spanned on the vertex set $S\subset V(H)$ in a graph $H$ is smooth then $T$ is a retract of $H$.
\end{lemma}

The following natural conjecture arises.

\begin{conjecture} If $H$ is a Sidorenko graph and a tree $T$ is a retract of $H$ then $T$ is smooth in $H$.
\end{conjecture}

We prove the following special case of the above conjecture.

\begin{theorem}\label{edgesmooth} If $H$ is a Sidorenko graph then every edge is smooth in $H$.
\end{theorem}

\begin{proof} Let $G$ be a finite graph with edge density $d$. We denote by $G_k$ the $k$-th tensor power of $G$.
Assume that $e$ is a fixed edge in $H$ and that $H$ has $a+1$ edges. Let $H^*$ be the graph obtained from $H$ by removing the edge $e$ and labeling the two endpoints by $\{1,2\}$.
Let $E_k$ denote the set of edges $f$ in $G_k$ for which $t_f(H^*,G_k)\geq d^{ka}/2$.

Let $\Omega_k$ be the probability space of a randomly chosen edge in $G_k$ and let $X_k$ be the random variable $t_f(H^*,G_k)$ on $\Omega_k$. It is clear from the definition of $G_k$ that the distribution of $X_k$ is the product of $k$ independent copies of $X_1$. Let $\epsilon>0$ be an arbitrary number. By the law of large numbers, if $k$ is big enough then $$\mathbb{P}(|(\ln X_k)/k-\mathbb{E}(\ln X_1)|>\epsilon)\leq\epsilon.$$

Now we estimate the probability $\mathbb{P}(E_k)=|E_k|/|E(G_k)|$ in the probability space $\Omega_k$.
Let $G'_k$ denote the graph obtained from $G_k$ by deleting the edge set $E_k$.
The edge density of $G'_k$ is equal to $d^k(1-\mathbb{P}(E_k))$ and so using the fact that $H$ is Sidorenko it follows that $t(H,G'_k)\geq d^{(a+1)k}(1-\mathbb{P}(E_k))^{(a+1)}$. On the other hand $t_f(H^*,G'_k)\leq d^{ka}/2$ for every edge $f$ in $G'_k$ and thus $t(H,G'_k)\leq d^k(1-\mathbb{P}(E_k))d^{ka}/2$. We obtain that $1/2\geq (1-\mathbb{P}(E_k))^a$ and thus $\mathbb{P}(E_k)\geq 1-2^{-a}>0$.

Notice that the lower bound for $\mathbb{P}(E_k)$ does not depend on $k$. It follows that if $\epsilon<1-2^{-a}$ and $k$ is sufficiently big then the event $E_k$ intersect the event $|(\ln X_k)/k-\mathbb{E}(\ln X_1)|\leq\epsilon$ and thus
$$\mathbb{E}(\ln X_1)\geq \ln(d^{ka}/2)/k-\epsilon$$ holds for all such choices of $\epsilon$ and $k$.
Consequently $\mathbb{E}(\ln X_1)\geq a\ln d$. This inequality is equivalent with the smoothness of $e$.
\end{proof}

\medskip

\noindent{\it Proof of theorem \ref{edglue}.}~The statement is a direct consequence of theorem \ref{edgesmooth} and lemma \ref{gluing}.

\end{document}